\documentclass[pdflatex,sn-mathphys-num]{sn-jnl}

\usepackage{graphicx}%
\usepackage{multirow}%
\usepackage{amsmath,amssymb,amsfonts}%
\usepackage{amsthm}%
\usepackage{mathrsfs}%
\usepackage[title]{appendix}%
\usepackage{xcolor}%
\usepackage{textcomp}%
\usepackage{manyfoot}%
\usepackage{booktabs}%
\usepackage{algorithmicx}%
\usepackage{algpseudocode}%
\usepackage{listings}%
\usepackage{algorithm}%
\usepackage{hyperref}
\usepackage{bookmark}

\theoremstyle{thmstyleone}%
\newtheorem{theorem}{Theorem}[section]
\newtheorem{proposition}[theorem]{Proposition}%
\newtheorem{assumption}{Assumption}[section]
\newtheorem{lemma}{Lemma}[section]

\theoremstyle{thmstyletwo}%
\newtheorem{remark}{Remark}[section]%

\theoremstyle{thmstylethree}%
\allowdisplaybreaks

\begin{document}

\title[Article Title]{1/2 order convergence rate of Euler-type methods for time-changed stochastic differential equations with super-linearly growing drift and diffusion coefficients}


\author[1,2]{\fnm{Shuai} \sur{Wang}}\email{swang767@connect.hkust-gz.edu.cn}

\author*[1]{\fnm{Yuanling}  \sur{Niu}}\email{yuanlingniu@csu.edu.cn}
\author[2]{\fnm{Ying} \sur{Zhang}}\email{yingzhang@hkust-gz.edu.cn}
\affil[1]{\orgdiv{School of Mathematics and Statistics}, \orgname{Central South University}, \orgaddress{\street{No.932 South Lushan Road}, \city{Changsha}, \postcode{410083}, \country{China}}}

\affil[2]{\orgdiv{Financial Technology Thrust, Society Hub}, \orgname{The Hong Kong University of Science and Technology (Guangzhou)}, \orgaddress{\street{No.1 Duxue Road}, \city{Guangzhou}, \postcode{511453}, \country{China}}}

\abstract{This paper investigates the strong convergence properties of two Euler-type methods for a class of time-changed stochastic differential equations (TCSDEs)  with super-linearly growing drift and diffusion coefficients. Building upon existing research, we propose a backward Euler method (BEM)  and introduce  its  explicit counterpart---the projected Euler method (PEM). We prove that both methods converge strongly in the $L_2$-sense at the optimal rate of 1/2. This result extends the applicability of both the BEM and the PEM to a broader class of TCSDEs. Moreover, the two methods offer complementary strengths: while BEM possesses wide applicability, PEM is computationally more efficient. Numerical simulations confirm our theoretical findings and illustrate practical performance of both schemes.}

\keywords{Time-changed stochastic differential equations, Backward Euler method, Projected Euler method, Strong convergence, Super-linear coefficients}


\pacs[MSC Classification]{60H10 $\cdot$ 65C30 $\cdot$ 60J60}

\maketitle
\section{Introduction}
Stochastic differential equations (SDEs) are widely used to model diverse phenomena in physics, chemistry, finance, and engineering \cite{MR1926011,MR2069903,MR1783083}. As a specialized subclass of SDEs, time-changed stochastic differential equations (TCSDEs) replace the standard  temporal variable $t$ in classical SDEs with a time-changed process $E(t)$. This transformation extends their capability to model anomalous diffusion processes \cite{MR3623093,MR1882830,MR2841938}. In particular, when the time-changed process is chosen as the inverse of an $\alpha$-stable subordinator with $0<\alpha<1$, the probability density functions of the solutions to TCSDEs satisfy a fractional Fokker-Planck equations (FFPE). This FFPE, characterized by fractional time derivatives, governs the anomalous diffusion dynamics inherent to the time-changed process \cite{MR2783336,MR2736349,MR3753615}.

Since TCSDEs can be viewed as SDEs driven by semimartingales, \cite{MR2822482}~ established the result on the existence and uniqueness of solutions to TCSDEs using tools from semimartingale theory. Crucially, \cite[Theorem 4.2]{MR2822482}~derived a duality principle, which connects the TCSDE of the form 
\begin{equation}\label{fc0}
	\mathrm{d}X(t)=b(E(t),X(t))\mathrm{d}E(t)+g(E(t),X(t))\mathrm{d}W(E(t))
\end{equation}
with the classical SDE given by   
$$
\mathrm{d}Y(t)= b(t, Y(t))\mathrm{d}t+ g(t, Y(t)) \mathrm{d}W(t),
$$
where $X(0) =Y(0)$ and $\left\{W(t)\right\}_{t\geq0}$ is the standard Brownian motion.

Due to the presence of the time-changed process $E(t)$, obtaining analytic solutions to TCSDEs is a challenging task. We therefore consider to obtain their numerical approximations. For TCSDE~\eqref{fc0}, by leveraging the aforementioned duality principle, \cite{MR3593021}~established strong and weak convergence results for the Euler method under global Lipschitz conditions of the coefficients. Subsequently, \cite{MR4179718}~and \cite{MR4754794}~proposed respectively a backward Euler method (BEM) and a split-step theta method to approximate TCSDE~\eqref{fc0} and extended the convergence results in \cite{MR3593021} to the case with a super-linearly growing drift coefficient. All these works achieved  strong convergence rate of order $1/2$ for their proposed numerical methods. In \cite{MR4065764}, a truncated Euler method is proposed to approximate TCSDE~\eqref{fc0} with super-linearly growing drift and diffusion coefficients. A strong convergence result is established, but with the rate of convergence equal to ~$1/2 - \epsilon$ with~$\epsilon>0$. Furthermore, for TCSDEs that are  not of form~\eqref{fc0}, including, e.g, the time-changed McKean Vlasov SDEs and the TCSDEs given by
$$
\mathrm{d}X(t)=b(t,X(t))\mathrm{d}E(t)+g(t,X(t))\mathrm{d}W(E(t)),
$$
the duality principle is no longer applicable, hence, new approaches need to be considered to approximate TCSDEs in general forms. We refer  to~\cite{MR3958021,MR4292449,MR4556404, Liu2023AMM, MR4595993,MR4536712} for more details.

In this paper, we focus on TCSDE~\eqref{fc0}. Our goal is to approximate the solution of TCSDE~\eqref{fc0} with super-linearly growing drift and diffusion coefficients. To this end, we propose two numerical schemes, namely, BEM and projected Euler method (PEM), and establish their convergence results in the strong $L_2$-sense with the optimal rate of convergence equal to 1/2. To the best of the authors' knowledge, these are the first such results in the literature. We further conduct numerical experiments to demonstrate the effectiveness of the proposed methods and to highlight the distinct advantages of BEM and PEM: while BEM can be applied to a broader range of TCSDEs and has better behavior in stiff TCSDEs, PEM is more computationally efficient.

The rest of this paper is structured as follows. Section~\ref{sec2} introduces the relevant notation and necessary assumptions. Section~\ref{sec3} presents the approximation of the time-changed process, our proposed schemes, and their convergence results. Section~\ref{sec4} reports the numerical experiments validating our theoretical findings. Finally, Section~\ref{sec5} provides proofs supporting the main results in Sections~\ref{sec3}, while Appendices~\ref{app2} and~\ref{app1} contain proofs for the auxiliary results in Sections~\ref{sec4} and~\ref{sec5}.
\section{Preliminaries}
\label{sec2}
\subsection{Notation}
Throughout this paper, we work on a complete probability space~$(\Omega, \mathcal{F}, \mathbb{P})$ endowed with a filtration $\mathbb{F}=\left\{\mathcal{F}_t\right\}_{t \geqslant 0}$ satisfying the usual conditions. Let $\left\{W(t)\right\}_{t\geq0}$ be an $m$-dimensional $\mathbb{F}$-adapted standard Brownian motion defined on~$(\Omega, \mathcal{F}, \mathbb{P})$. Let $d,m>0$, we denote by $\langle x, y \rangle$ the inner product of vectors $x, y\in\mathbb{R}^d$ and denote by $|x|$ the Euclidean norm of $x\in\mathbb{R}^d$. For any $a,b\in\mathbb{R}$, we denote by $a \wedge b$ and $a \vee b$ their minimum and maximum, respectively. We denote by~$\mathbb{N}$ the set of all positive integers and $\mathbb{N}_0:=\mathbb{N}\cup\{{0}\}$. For a matrix $A\in\mathbb{R}^{m\times d}$, the trace norm is denoted by $|A|=\sqrt{trace(A^TA)}$ and we write $\|A\|$ to denote the operator norm of $A$.

Similar to \cite{MR4179718,MR3593021,MR4065764}, we denote by~$\left\{D(t)\right\}_{t\geq0}$ an~$\mathbb{F}$-adapted subordinator with Laplace exponent $\psi$ and L\'evy measure $\nu$. More precisely, $\left\{D(t)\right\}_{t\geq0}$ is a one-dimensional nondecreasing L\'evy process with c\`adl\`ag paths starting at zero. Its Laplace transform is given by
$$
\mathbb{E}\left[e^{-s D(t)}\right]=e^{-t \psi(s)},
$$
where $s>0$, $\psi(s)=a s+\int_0^{\infty}\left(1-e^{-s x}\right) \nu(\mathrm{d} x)$  with $a \geqslant 0$, and~$\int_0^{\infty}(x \wedge 1) \nu(d x)<\infty$. We assume that the L\'evy measure is infinite, i.e., $\nu(0, \infty)=\infty$, which implies that $\left\{D(t)\right\}_{t\geq0}$ has strictly increasing paths with infinitely many jumps. Denote by~$\left\{E(t)\right\}_{t\geq0}$ the inverse subordinator of~$\left\{D(t)\right\}_{t\geq0}$, which is defined as 
$$
	E(t):=\inf \{u>0 ; D(u)>t\},
$$
for all~$t\geq0$, which has continuous and nondecreasing paths \cite{MR2822482}. We always assume that~$\left\{D(t)\right\}_{t\geq0}$ and $\left\{W(t)\right\}_{t\geq0}$ are independent. The process $\left\{W(E(t))\right\}_{t\geq0}$ is usually called a time-changed Brownian motion. 
\subsection{Setting and assumptions}
We consider the TCSDE $\{X_t\}_{t\geq0}$ given by 
\begin{equation}\label{fc}
	\mathrm{d}X(t)=b(E(t),X(t))\mathrm{d}E(t)+g(E(t),X(t))\mathrm{d}W(E(t)),
\end{equation}
where the initial value $X(0)$ has a finite $p$-th moment with $p>0$, $b:[0,\infty)\times\mathbb{R}^d \to\mathbb{R}^d$ and $g:[0,\infty)\times\mathbb{R}^d \to\mathbb{R}^{d\times m}$ are measurable functions, and $\left\{W(t)\right\}_{t\geq0}$ is an $m$-dimensional standard Brownian motion.
Then, we impose the following assumptions on the coefficients of TCSDE \eqref{fc}
\begin{assumption}\label{A1}
	There exist  $K>0$ and $\gamma>1$ such that, for all $x, y \in \mathbb{R}^d$ and $t \in[0,\infty)$,
	$$		
	|b(t, x)-b(t, y)| \leq K(1+|x|^{\gamma-1}+|y|^{\gamma-1})|x-y|,\\
	$$
	and $$ |b (t, 0)|\vee|g (t, 0)|\leq K.$$ 
\end{assumption}
\begin{assumption}\label{A2}
	There exist  $K_1>0$ and $\eta>1$ such that, for all $x, y \in \mathbb{R}^d$ and $t \in[0,\infty)$,
	$$
	\left<b(t, x)-b(t, y),x-y\right>+\frac{\eta-1}{2}|g(t, x)-g(t, y)|^2 \leq 
	K_1|x-y|^2 .
	$$
\end{assumption}
\begin{assumption}\label{A3}
	There exist  $K>0$ and $p^\star>1 	$ such that, for all $x\in \mathbb{R}^d$ and $ t \in[0,\infty)$,
	$$
	\left<b(t, x),x\right>+\frac{p^\star-1}{2}|g(t, x)|^2 \leq 
	K(1+|x|^2).
	$$
\end{assumption}
\begin{assumption}\label{A4}
	There exist  $K>0$ and $\gamma>1$ such that, for all $x \in \mathbb{R}^d$ and $s,t \in[0,\infty)$,
	$$
	|b(t, x)-b(s, x)|\vee|g(t, x)-g(s, x)| \leq K(1+|x|^\gamma)|t-s|^\frac{1}{2}.
	$$
\end{assumption}
\begin{remark}\label{re2.1}
	By Assumptions \ref{A1} and \ref{A2}, we deduce that, for all $x, y \in \mathbb{R}^d$ and $ t \in[0,\infty)$,
	$$
	|b(t, x)| \leq K(1+|x|^\gamma) ,\\
	$$
	$$
	|g(t, x)-g(t, y)|^2 \leq K(1+|x|^{\gamma-1}+|y|^{\gamma-1})|x-y|^2 , \\
	$$
	$$
	|g(t, x)| \leq K(1+|x|^{\frac{\gamma+1}{2}}) ,\\
	$$
	where $K>0$.
\end{remark}
\begin{remark}
	Under Assumptions \ref{A1}-\ref{A4}, TCSDE \eqref{fc} admits a unique strong  solution, see \cite[Lemma 4.1]{MR2822482} and \cite[Chapter V]{MR2020294}.
\end{remark}
According to \cite[Theorem 4.2]{MR2822482}, there is a relationship between TCSDE \eqref{fc} and the classical SDE given by
\begin{equation}\label{do}
	\mathrm{d}Y(t)= b(t, Y(t))\mathrm{d}t+ g(t, Y(t)) \mathrm{d}W(t),
\end{equation} 
with $Y(0)=X(0)$. This relationship is referred as the duality principle described below, which plays a central role in our analysis. 
\begin{lemma}(Duality principle)\label{le2.3}
	Under Assumptions \ref{A1}-\ref{A4}, if $\{Y (t)\}_{t\geq0}$ is the unique  strong solution to SDE \eqref{do}, then the time-changed process $\{Y (E(t))\}_{t\geq0}$ is the unique strong solution to TCSDE \eqref{fc}. Conversely, if $\{X (t)\}_{t\geq0}$ is the unique strong solution to TCSDE \eqref{fc}, then the process $\{X (D(t))\}_{t\geq0}$ is the unique strong solution to SDE \eqref{do}.
\end{lemma}
\section{Numerical schemes and main results}
\label{sec3}
To approximate the solution of TCSDE~\eqref{fc}, we apply the duality principle in Lemma~\ref{le2.3} and consider the numerical approximation of \(\{Y(E(t))\}_{t\geq0}\). This, however, involves the discretization of both \(\{Y(t)\}_{t\geq0}\) and \(\{E(t)\}_{t\geq0}\). In this section, we first describe the approach adopted to discretize \(\{E(t)\}_{t\in[0,T]}\) with a fixed $T>0$, and then introduce two numerical methods, i.e., BEM and PEM, to approximate \(\{Y(t)\}_{t\geq0}\). We establish, in Theorems~\ref{BEMT} and \ref{PEM1} below, the strong convergence results for the proposed schemes with the optimal rate of convergence equal to 1/2. We defer the proofs for the aforementioned results to Section \ref{sec5}.

\subsection{Approximation of time-changed process}
Fix $T>0$, to approximate the time-changed process $\left\{E(t)\right\}_{t\in[0,T]}$, we consider the discretization adapted from \cite{PhysRevE.82.011117,MR3593021}. For a fixed step size $h \in (0,1]$ and $i \in \mathbb{N}$, we define iteratively
\[
D(i h) := D\left((i-1)h\right) + Z_i, 
\]
where $\{Z_i\}_{i\in\mathbb{N}}$ is an i.i.d.\ sequence with $Z_i$ having the same distribution as $D(h)$. The procedure terminates at the smallest integer $N\in \mathbb{N}_0$ satisfying
\begin{align}\label{NandT}
	T \in [D(Nh), D((N+1)h)).
\end{align}
Then, we define, for any $t\in[0,T]$, that
\begin{equation}\label{Ed}
	E_h(t) := \Big( \min \big\{ n \in \mathbb{N}_0 : D(n h) > t \big\} - 1 \Big) h.
\end{equation}
This implies $E_h(t)=nh$, for any $t\in [D(nh), D(n+1)h)$, $n=0,\dots N$, and, in particular, $E_h(T)=Nh$.

\subsection{BEM with convergence result}

In this section, we adapt BEM, originally proposed in \cite{MR4179718}, to approximate TCSDE \eqref{fc} with super-linear drift and diffusion coefficients. Here, we follow the analytical framework of \cite{cai2024convergence, MR4132905} to extend the method.

Define a temporal grid $\{t_n = n h, n \in \mathbb{N}_0\}$ with $h\in(0,1]$. The BEM scheme for SDE \eqref{do} is given by, for any $n\in\mathbb{N}_0$,
\begin{equation}\label{BEM}
	Y_h^B(t_{n+1}) = Y_h^B(t_n) + b\left(t_{n+1}, Y_h^B(t_{n+1})\right)h + g\left(t_n, Y_h^B(t_n)\right)\Delta W_{n}
\end{equation}
with $Y_h^B(0) = Y(0)$ being an $\mathbb{R}^d$-valued random variable and $\Delta W_{n} := W(t_{n+1}) - W(t_n)$.
\begin{remark}
	Under Assumption \ref{A2}, the BEM scheme \eqref{BEM} is well-defined for all $h \in (0, 1/2K_1)$, see \cite{MR3608309}. 
\end{remark} 
By Lemma~\ref{le2.3}, our proposed BEM for TCSDE \eqref{fc}, denoted by $\left\{Y_h^B(E_h(t))\right\}_{t\in[0,T]}$, is defined using \eqref{Ed} and \eqref{BEM}, which is given by
\begin{equation}\label{BEMTCSDE}
	Y_h^B(E_h(t))=Y_h^B(t_n),
\end{equation}
where $t\in [D(nh), D(n+1)h)$ and $n\in\{0,1\cdots N\}$.
\begin{remark}
	Unlike \cite{MR4179718,MR3593021,MR4065764}, we do not introduce a continuous-time interpolation of BEM \eqref{BEM}. This is due to the fact that only~$\left\{Y_h^B(E_h(t))\right\}_{n\in\{0,1\cdots N\}}$ defined in \eqref{BEMTCSDE} is used in the convergence analysis, see, e.g. \eqref{RHS} in Section \ref{sec5.1} below. 
\end{remark}
\begin{remark}
	Throughout this paper,  we use $C$ to denote generic positive constants whose values may vary across different occasions, but are always independent of $h,t,n,$ and $T$.
\end{remark}

We provide below a strong convergence result for BEM \eqref{Ed}-\eqref{BEMTCSDE} with the optimal rate of convergence.
\begin{theorem}\label{BEMT}
	Suppose that Assumptions \ref{A1}-\ref{A4} hold with $\eta>2$, $p^\star \geq4\gamma-2$. Let $T>0$, and $h\in(0,1\wedge\frac{1}{4K_1}]$ be fixed. Then, BEM \eqref{Ed}-\eqref{BEMTCSDE} converges in $L_2$ to the exact solution of TCSDE \eqref{fc} with order 1/2, i.e., for all $t\in[0,T]$,
	$$
	\mathbb{E}\left[\left|X(t)-Y^B_h(E_h(t))\right|^2\right] \leq Che^{CT}.
	$$
\end{theorem}
\begin{remark}
	The rate of convergence obtained in Theorem \ref{BEMT}, as well as Theorem \ref{PEM1}, is optimal, since it is limited by the H\"{o}lder continuity of the exact solution which is inherently of order 1/2. We refer to Section~\ref{sec5} for more details.
\end{remark}
\subsection{PEM with convergence result}
In this section, we propose an alternative PEM method to approximate TCSDE \eqref{fc}, which is based on the framework developed in \cite{MR3493491}. 

\label{sec3.1}
Fixing $h\in(0,1]$, we define, for any $x\in\mathbb{R}^d$, 
\begin{equation}\label{xq}
	\begin{aligned}
		\kappa(x):=\min \left(1, h^{-\alpha}|x|^{-1}\right)x,
	\end{aligned}
\end{equation}
where $\alpha=\frac{1}{2(\gamma-1)}$. For a temporal grid $\{t_n = n h, n \in \mathbb{N}_0\}$, the PEM scheme for SDE \eqref{do} is given by, for any $n\in \mathbb{N}_0$,
\begin{align}\label{PEM}
	Y_h^P\left(t_{n+1}\right):= & \kappa(Y_h^P(t_n))+h b\left(t_n, \kappa(Y_h^P(t_n))\right) 
	+ g\left(t_n, \kappa(Y_h^P(t_n))\right)\Delta W_n
\end{align}
with $Y_h^P(0) = Y(0)$ being an $\mathbb{R}^d$-valued random variable. Consequently, by Lemma~\ref{le2.3}, PEM for TCSDE \eqref{fc} can be obtained using \eqref{Ed}, \eqref{xq}, \eqref{PEM}, which is given by
\begin{equation}\label{PEMTCSDE}
	Y_h^P(E_h(t))=Y_h^P(t_n),
\end{equation}
where $t\in [D(nh), D(n+1)h)$ and $n\in\{0,1\cdots N\}$.
\begin{theorem}\label{PEM1}
	Suppose Assumptions \ref{A1}-\ref{A4} hold with $\eta>3$, $p^\star \geq6\gamma-4$. Let $T>0$ and $h\in(0,1]$ be fixed. Then, PEM defined by \eqref{Ed}, \eqref{xq}-\eqref{PEMTCSDE} converges in $L_2$ to the exact solution of TCSDE \eqref{fc} with order 1/2, i.e., for all $t\in[0,T]$,
	$$
	\mathbb{E}\left[\left|X(t)-Y_h^P(E_h(t))\right|^2\right] \leq Che^{CT}.
	$$ 
\end{theorem}
\begin{remark}
	A comparison of Theorems \ref{BEMT} and \ref{PEM1} reveals that BEM imposes weaker conditions on $\eta$ and $p^\star$ in Assumptions~\ref{A2} and \ref{A3}, implying its broader applicability to TCSDEs compared to PEM. However, BEM is more computationally expensive than PEM as illustrated by our numerical experiments in Section \ref{sec4}, which can be explained by the fact that BEM is an implicit method while PEM is an explicit method.  
\end{remark}
\section{Numerical experiments}\label{sec4}
In this section, we present two examples to validate our theoretical results in Theorems \ref{BEMT} and \ref{PEM1}. Throughout, we define $\{E(t)\}_{t\geq0}$ as a $0.9$-stable inverse subordinator whose Bernstein function is given by $\phi(r) = r^{0.9}$, and then use \cite[Algorithm 1]{jum2015numerical} to generate $\{D(t)\}_{t\geq0}$. Therefore, for a given step size $h$, we can obtain $\{E_h(t)\}_{t\in[0,T]}$ by applying \eqref{Ed}.
\subsection{One-dimensional non-linear TCSDE}\label{li1}
Fixing $T>0$, we consider using BEM \eqref{Ed}-\eqref{BEMTCSDE} and PEM \eqref{Ed}, \eqref{xq}-\eqref{PEMTCSDE} to approximate the TCSDE $\{X(t)\}_{t\in[0,T]}$ given by
\begin{equation}\label{fc1}
	\mathrm{d}X(t)=(X^2(t)-2X^5(t))\mathrm{d}E(t)+X^2(t)\mathrm{d}W(E(t))
\end{equation}
with $X(0)=1$.
\begin{proposition}\label{E1p}
	TCSDE \eqref{fc1} satisfies Assumptions \ref{A1}-\ref{A4}.
	\begin{proof}
		See Appendix \ref{app2}.
	\end{proof}
\end{proposition}

By Proposition \ref{E1p}, Theorems \ref{BEMT} and \ref{PEM1} can be applied to ensure the convergence of BEM and PEM. In particular, the step sizes we choose below satisfy the special condition $h\in(0,1\wedge\frac{1}{4K_1}]$ in Theorem ~\ref{BEMT} for BEM. Moreover, both methods for solving \eqref{fc1} are expected to achieve a convergence rate of $1/2$.

To obtain numerical approximations using BEM and PEM, we follow the approach in \cite{MR4179718,MR4065764}. More precisely, we first approximate the duality equation $\{Y(t)\}_{t\in[0,T]}$ associated with TCSDE \eqref{fc1} given by
$$
\mathrm{d}Y(t)=(Y^2(t)-2Y^5(t))\mathrm{d}t+Y^2(t)\mathrm{d}W(t)
$$
with $Y(0)=1$. Recall $N$ is found by \eqref{NandT}. Let $\{Y_h(t_n)\}_{n\in\{0,1\cdots N\}}$ denote the numerical approximation to the true solution $\{Y(t_k)\}_{k\in\{0,1\cdots N\}}$, which can be viewed as the numerical approximation of the solution to TCSDE \eqref{fc1} via $Y_h(E_h(t))=Y_h(t_n)$ defined in \eqref{BEMTCSDE} and \eqref{PEMTCSDE}. Then, we compute the strong $L_2$ error at time $T>0$ using
\begin{equation}\label{MC}
	\left(\frac{1}{M}\sum_{j=1}^{M}\left|Y^j_{h_0}(E_{h_0}(T))-Y^j_h(E_h(T))\right|^2\right)^\frac{1}{2},
\end{equation}
where $Y^j_{h_0}(E_{h_0}(T))$ denotes the reference solution obtained using numerical schemes with $h_0=2^{-15}$, $Y^j_h(E_h(T))$ denotes the numerical approximations for $h\in\{2^{-9},2^{-8},2^{-7},2^{-6}\}$, and $M$ is the number of sample paths.

For numerical experiments, we set $T=1$ and $M=300$. Figure \ref{paths}a shows a single sample path of $\{D(t)\}_{t\in[0,1]}$ and $\{E(t)\}_{t\in[0,1]}$ simulated using $h_0=2^{-15}$ while Figure \ref{paths}b depicts a sample path for each of $\{X(t)\}_{t\in[0,1]}$ and $\{Y(t)\}_{t\in[0,1]}$. The fluctuation of $\{X(t)\}_{t\in[0,1]}$ is smaller than that of $\{Y(t)\}_{t\in[0,1]}$ as the presence of $\{E(t)\}_{t\in[0,1]}$ dampens its fluctuation. This observation explains why TCSDEs can be used to model anomalous diffusion processes. Figure \ref{sll1} presents the convergence results for both BEM and PEM. The black line is the reference line with slope 1/2, whereas the red and blue  lines represent the log$L_2$ error \eqref{MC} for BEM and PEM, respectively. In particular, the estimated convergence rates obtained using linear regression are $0.4955$ for BEM and $0.5742$ for PEM, which support our  theoretical results. Furthermore, in Table \ref{runtime}, we report the runtimes for BEM and PEM with different step sizes, which exclude the generation time of $\{E(t)\}_{t\in[0,1]}$. 
We note that PEM requires significantly less computational time than BEM among all cases.
\begin{figure}[!ht]
	\centering
	\begin{minipage}[b]{0.48\textwidth}
		\centering
		\includegraphics[width=0.95\textwidth]{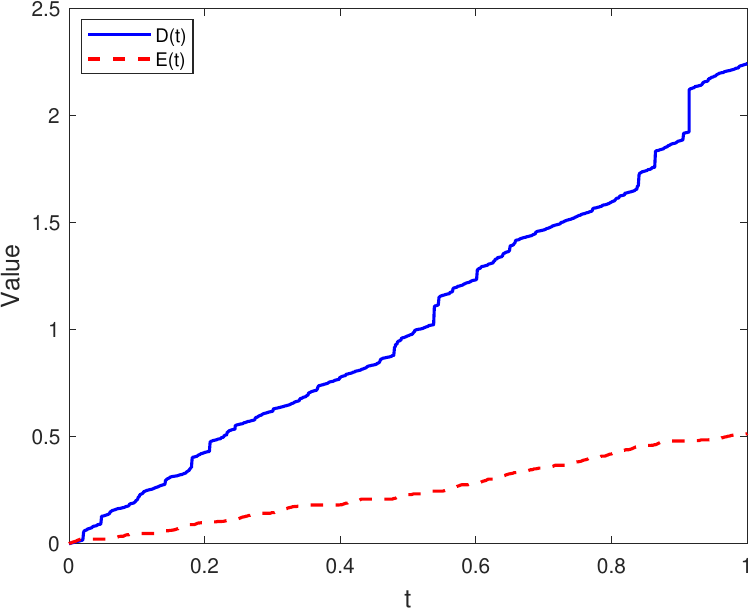}  
		\par\vspace{2mm}
		{\small\centering (a) One path of $D(t)$ and $E(t)$}
	\end{minipage}
	\hfill
	\begin{minipage}[b]{0.48\textwidth}
		\centering
		\includegraphics[width=0.95\textwidth]{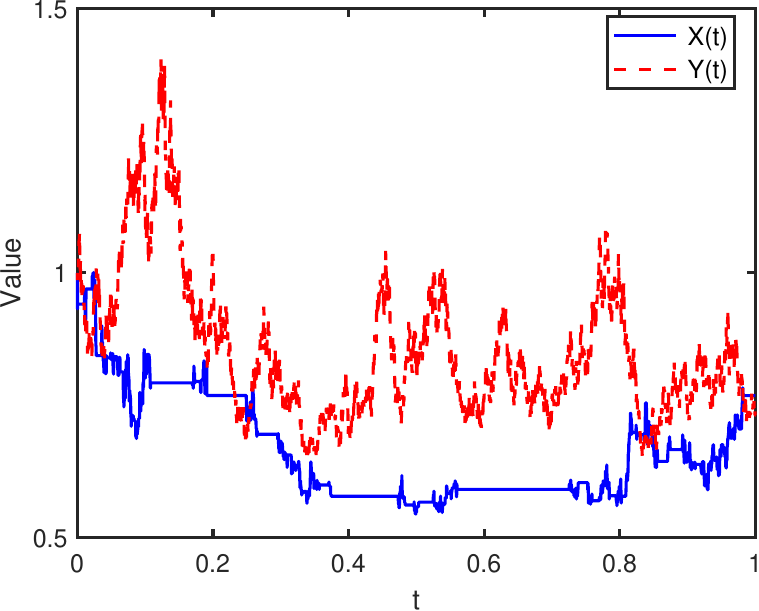}
		\par\vspace{2mm}
		{\small\centering (b) One path of $X(t)$ and $Y(t)$}
	\end{minipage}
	
	\vspace{4mm}
	
	\caption{Simulated sample paths}
	\label{paths}  
\end{figure}
\noindent  
\begin{figure}[!ht]
	\centering
	\includegraphics[width=100mm]{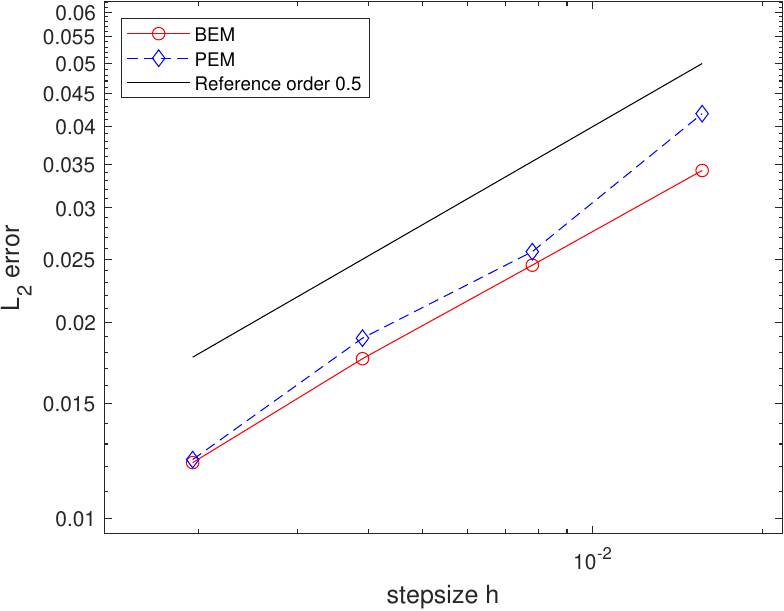}
	\caption{$L_2$ errors between the exact solution and numerical solutions obtained using  BEM and PEM with step sizes $h \in\{2^{-6},2^{-7},2^{-8},2^{-9}\}$}
	\label{sll1}
\end{figure}
\begin{table}[tbp]
	\centering
	\caption{Comparison of BEM and PEM methods: $L_2$ errors and runtimes at different step sizes}
	\label{runtime}
	\setlength{\tabcolsep}{15pt}
	\vspace{0.5em}  
	\begin{tabular}{c cc cc}
		\toprule
		\multirow{2}{*}{Step size} & \multicolumn{2}{c}{BEM} & \multicolumn{2}{c}{PEM} \\
		\cmidrule(r){2-3} \cmidrule(l){4-5}
		& $L_2$ error & Runtime & $L_2$ error & Runtime \\
		\midrule
		$2^{-9}$  & 0.008 & 0.9375 & 0.009 & 0.2300 \\
		$2^{-8}$  & 0.010 & 0.7963 & 0.014 & 0.0630 \\
		$2^{-7}$  & 0.016 & 0.3125 & 0.018 & 0.0200 \\
		$2^{-6}$  & 0.032 & 0.1526 & 0.040 & 0.0050 \\
		\bottomrule
	\end{tabular}
\end{table}

\subsection{2D stiff TCSDE}
Fixing $T>0$, we consider a 2-dimensional stiff TCSDE $\{X(t)\}_{t\in[0,T]}$ given by
\begin{equation}\label{li2}
	\mathrm{d}X(t)=(f(X(t))-AX(t))\mathrm{d}E(t)+g(X(t))\mathrm{d}W(E(t))
\end{equation}
with $X(0)=X_0\in\mathbb{R}^2$, where $A$ is a positive symmetric matrix defined by
\begin{equation}\label{E2A}
	A=\frac{1}{2}\left[
	\begin{matrix}
		201 & -199\\
		-199 & 201
	\end{matrix} 
	\right],
\end{equation}
and $f:\mathbb{R}^2\rightarrow \mathbb{R}^2$, $g:\mathbb{R}^2\rightarrow \mathbb{R}^{2\times2}$ are measurable functions given by 
\begin{equation}\label{E2FG}
	f(x)=\left[
	\begin{matrix}
		x_1-x_1^3\\
		x_2-x_2^3
	\end{matrix} 
	\right], \qquad
	g(x)=\frac{1}{2}\left[
	\begin{matrix}
		x_1 & 0\\
		0 &x_2
	\end{matrix} 
	\right].
\end{equation}
The system stiffness of TCSDE \eqref{li2}-\eqref{E2FG} originates from the large eigenvalue of matrix $A$, and we refer to  \cite{MR3608309,MR3974743} for more details.
\begin{proposition}\label{E2p}
	TCSDE \eqref{li2}-\eqref{E2FG} satisfies Assumptions \ref{A1}-\ref{A4}.
	\begin{proof}
		See Appendix \ref{app2}.
	\end{proof}
\end{proposition}

By Proposition \ref{E2p}, the approximated solutions obtained using BEM and PEM converge strongly to the true solution of TCSDE~\eqref{li2}-\eqref{E2FG}. 

For the numerical experiment, we follow the same procedure as described in  Section \ref{li1}. The reference solution in this case is computed with step size $h_0 = 2^{-16}$, while the numerical approximations are calculated using $h = 2^{-i}$ for $i \in\{7,8,9,10,11\} $ or $h = 2^{-i}$ for $i \in \{5,6,7,8,9\}$. The other parameters, $T$ and $M$, remain the same as Section \ref{li1}.

We present simulation results in Figures~\ref{2wei} and \ref{e2}. More precisely, Figure~\ref{2wei} displays a sample path of the reference solution $X(t)=[X_1(t),X_2(t)]_{ t\in[0,1]}$. Figure~\ref{e2} reveals behavioral differences between two methods for TCSDE \eqref{li2}-\eqref{E2FG}. While BEM produces accurate approximations even with large step sizes, PEM approximations deviate significantly from the true solution in the $L_2$-sense when the step size is large.
These results suggest the superior stability of the implicit method (BEM) over the explicit method (PEM), consistent with classical numerical analysis theory.
\begin{figure}[!ht]
	\centering
	\begin{minipage}[b]{0.48\textwidth}
		\centering
		\includegraphics[width=\textwidth]{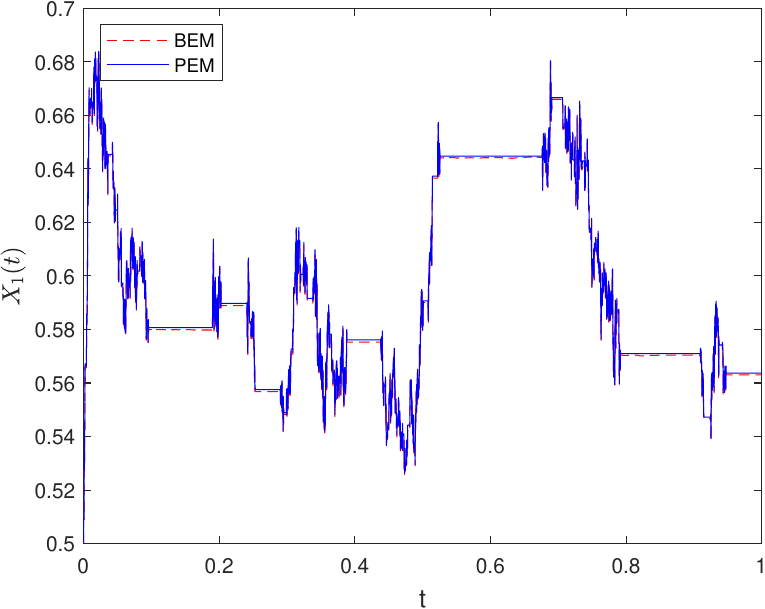}
		\par\vspace{2mm} 
		{\small\centering (a) One path of $X_1$}
	\end{minipage}
	\hfill 
	\begin{minipage}[b]{0.48\textwidth}
		\centering
		\includegraphics[width=\textwidth]{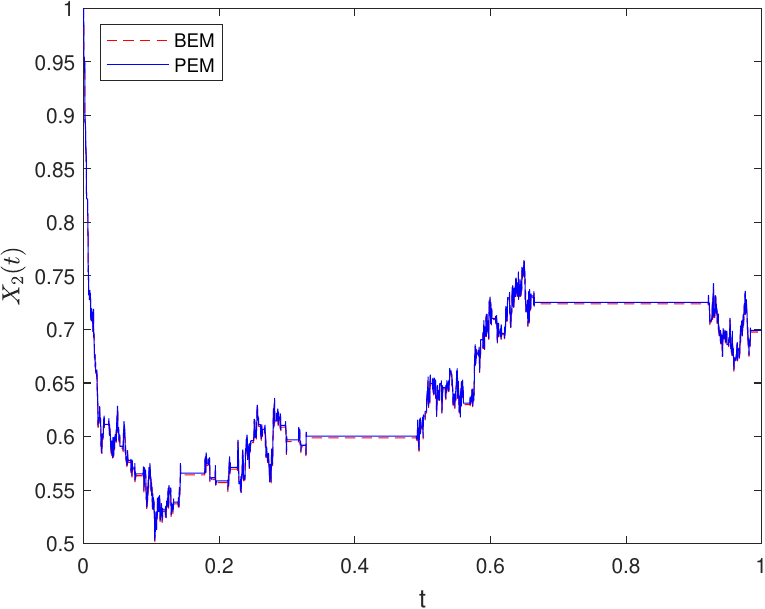}
		\par\vspace{2mm}
		{\small\centering (b) One path of $X_2$}
	\end{minipage}
	
	\vspace{4mm} 
	
	\caption{One path of reference solution $\{X(t)\}_{t\in[0,1]}=\{[X_1(t),X_2(t)]\}_{t\in[0,1]}$ obtained using BEM and PEM with $h_0=2^{-16}$}
	\label{2wei}
\end{figure}
\begin{figure}[!ht]
	\centering
	\begin{minipage}[b]{0.48\textwidth}
		\centering
		\includegraphics[width=\textwidth]{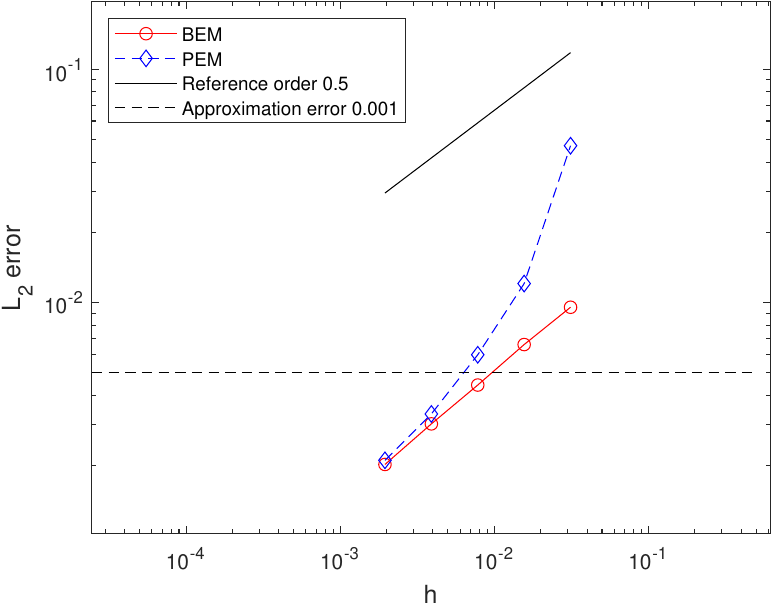}
		\par\vspace{2mm}
		{\small\centering (a) $h=2^{-i},i\in\{7,8,9,10,11\}$}
	\end{minipage}
	\hfill
	\begin{minipage}[b]{0.48\textwidth}
		\centering
		\includegraphics[width=\textwidth]{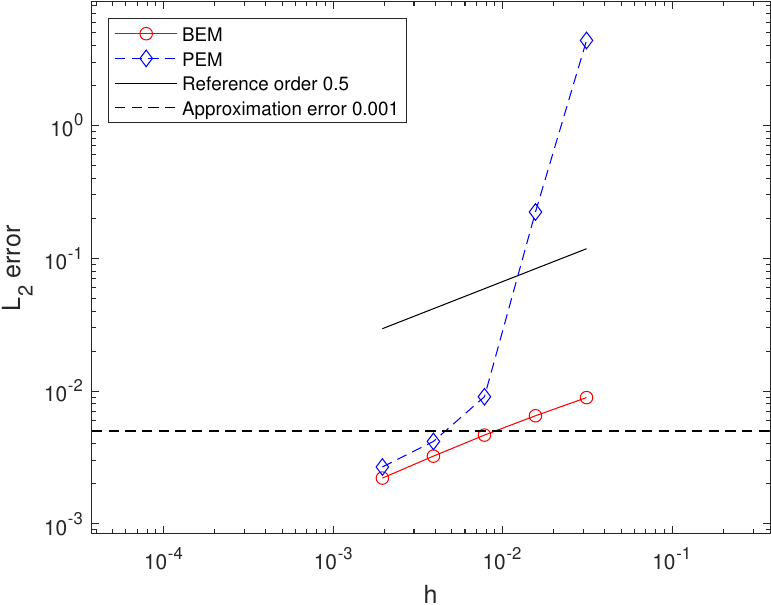}
		\par\vspace{2mm}
		{\small\centering (b) $h=2^{-i},i\in\{5,6,7,8,9\}$}
	\end{minipage}
	\vspace{4mm}
	\caption{Convergence results for BEM and PEM with different step sizes}
	\label{e2}
\end{figure}
\section{Proof of main results in Section \ref{sec3}}\label{sec5}
\subsection{Proof overview}\label{sec5.1}
For a fixed $T>0$, we denote by $\{Y_h(E_h(t))\}_{t\in[0,T]}$ the numerical approximation of TCSDE \eqref{fc} given explicitly by 
$$
Y_h(E_h(t))=\begin{cases}
	Y_h^B(E_h(t)) \text{\quad for BEM \eqref{Ed}-\eqref{BEMTCSDE}}, \\
	Y_h^P(E_h(t)) \text{\quad for PEM \eqref{Ed},\eqref{xq}-\eqref{PEMTCSDE}}. \\
\end{cases}
$$

To establish the strong convergence results in Theorems \ref{BEMT} and \ref{PEM1}, we aim to obtain an upper estimate for $\mathbb{E}\left[|X(t)-Y_h(E_h(t))|^2\right].$
To this end, We consider the following splitting: for any $t\in[0,T]$,
\begin{align}\label{RHS}
	\mathbb{E}\left[\left|X(t) - Y_h(E_h(t))\right|^2\right]\leq& 2\mathbb{E}\left[\left|X(t) - Y(E_h(t))\right|^2\right]+ 2\mathbb{E}\left[\left|Y(E_h(t)) - Y_h(E_h(t))\right|^2\right].
\end{align}
To upper bound the first term on the RHS of \eqref{RHS}, we use the properties of the time-changed process $\{E(t)\}_{t\in[0,T]}$ and SDE \eqref{do}. The relevant results are provided in Sections \ref{sec6.2} and \ref{sec6.3}, respectively. To upper bound the second term on the RHS of \eqref{RHS}, we investigate the convergence properties of BEM \eqref{Ed}-\eqref{BEMTCSDE} and PEM \eqref{Ed},\eqref{xq}-\eqref{PEMTCSDE}.
We present the convergence results and the corresponding proofs of Theorems \ref{BEMT} and \ref{PEM1} in Sections \ref{sec6.4} and \ref{sec6.5}, respectively.
\subsection{Error estimates for time-changed process}\label{sec6.2}
The following lemma states that $\{E_h(t)\}_{t\geq0}$ defined in \eqref{Ed} effectively approximates $\{E(t)\}_{t\geq0}$ for sufficiently small $h\in(0,1]$.
\begin{lemma}\label{shibian}\cite{MR3593021,MR2568272}
	We have, for any $h\in(0,1]$ and $t\geq0$, that
	\[
	E(t) - h  \leq  E_h(t)  \leq  E(t).
	\]
\end{lemma}

The next lemma provides an upper bound for the exponential moment of $\{E(t)\}_{t\geq0}$.
\begin{lemma}\label{qw}\cite[Lemma 2.3]{MR4179718}
	We have, for any $\lambda>0$ and $t \geq 0$, that
	$$
	\mathbb{E} \left[e^{\lambda E(t)}\right] \leq e^{C t}.
	$$
\end{lemma}
\subsection{Error estimates for SDE}\label{sec6.3}
To prove our main results, we present some pertinent results regarding SDE \eqref{do}. The result below provides moment estimates for the solution of SDE~\eqref{do}.
\begin{lemma}\cite[Theorem 4.1]{MR2380366}\label{ju}
	Suppose Assumptions \ref{A1}-\ref{A4} hold. Then, we have, for all $p\in[2,p^\star]$ and $t \geq 0$, that
	$$
	\mathbb{E}\left[\left|Y(t)\right|^p\right] \leq C e^{Ct}.
	$$
	Consequently, we obtain that 
	$$
	\sup_{0 \leq s \leq t}\mathbb{E}\left[\left|Y(s)\right|^p\right] \leq Ce^{Ct}.$$
\end{lemma}

The following lemma shows that the solution to SDE~\eqref{do} is H\"{o}lder continuous.
\begin{lemma}\label{jhd}
	Suppose Assumptions \ref{A1}-\ref{A4} hold. Then, we have, for any $p \in(1,\frac{p^\star} { \gamma}]$ and $t, s \geq 0$ with $|t-s| \leq 1$, that
	$$
	\mathbb{E}\left[\left|Y(t)-Y(s)\right|^p\right] \leq C|t-s|^{\frac{p}{2}} e^{C t}.
	$$
\end{lemma}
\begin{proof}
	See Appendix \ref{app1}.
\end{proof} 
\subsection{Proof of Theorem \ref{BEMT}}\label{sec6.4}
In this section, we provide the proof for Theorem \ref{BEMT}. We start by defining the local truncation error $\{R_i\}_{i\in\mathbb{N}_0}$ on a temporal grid $\{t_i=ih,i\in\mathbb{N}_0\}$:
	\begin{align}\label{R_i}
		R_i := \int_{t_{i-1}}^{t_{i}}  b(r,Y(r)) - b(t_i,Y(t_i))  \mathrm{d} r + \int_{t_{i-1}}^{t_{i}}  g(r,Y(r)) - g(t_{i-1},Y(t_{i-1}))  \mathrm{d} W(r).
	\end{align}
Then, we establish the relationship between the $L_2$ error and the local truncation error \eqref{R_i}.
\begin{proposition}\label{jfwc}
	Suppose Assumption \ref{A2} holds with $\eta > 2$ and $h\in(0,1\wedge\frac{1}{4K_1}]$. Then, for any $n\in\mathbb{N}_0$,
	\begin{small}
		$$
		\mathbb{E}\left[ \left|Y(t_n) - Y_h^B(t_n)\right|^2 \right] \leq C \left( \sum_{i=1}^{n} \mathbb{E}\left[\left|R_i\right|^2\right] + h^{-1} \sum_{i=1}^{n} \mathbb{E}\left[ \left| \mathbb{E}\left[R_i \mid \mathcal{F}_{t_{i-1}}\right] \right|^2 \right] \right) e^{C t_n}.
		$$
	\end{small}
\end{proposition}
\begin{proof}
	See Appendix \ref{app1}.
\end{proof}
\begin{lemma}\label{Ch}
	Suppose Assumptions \ref{A1}-\ref{A4} hold with $\eta > 2$, $p^\star \geq 4\gamma - 2$. Let $h\in(0,1]$. Then, for any $n\in\mathbb{N}_0$,
	$$
	\sum_{i=1}^{n} \mathbb{E}\left[\left|R_i\right|^2\right] + h^{-1} \sum_{i=1}^{n} \mathbb{E}\left[ \left| \mathbb{E}\left[R_i \mid \mathcal{F}_{t_{i-1}}\right]\right|^2 \right] \leq C h e^{C t_n}.
	$$
\end{lemma}
\begin{proof}
	See Appendix \ref{app1}.
\end{proof}

\begin{proof}[Proof of Theorem \ref{BEMT}]
	For any $t\in[0,T]$, applying Lemma \ref{le2.3} to \eqref{RHS}  yields
	\begin{align}\label{0}
		\mathbb{E}\left[\left|X(t)-Y_h^B(E_h(t))\right|^2\right]\leq& 2\mathbb{E}\left[\left|Y(E(t))-Y(E_h(t))\right|^2\right]+2 \mathbb{E}\left[\left|Y(E_h(t))-Y_h^B(E_h(t))\right|^2\right].		\end{align}
	To obtain an upper bound for the first term on the RHS of \eqref{0}, we note that Lemma \ref{shibian} ensures $|E(t)-E_h(t)|<h\leq1$, which indicates Lemma~\ref{jhd} is applicable. This, together with Lemma~\ref{qw}, yields
	\begin{equation}\label{1}
		\begin{aligned}
			\mathbb{E}\left[\left|Y(E(t))-Y(E_h(t))\right|^2\right]
			\leq Ch\mathbb{E}\left[e^{CE(t)}\right]\leq Che^{Ct}\leq Che^{CT}.
		\end{aligned}
	\end{equation}
	To establish an upper bound for the second term on the RHS of \eqref{0}, we use Proposition \ref{jfwc} and Lemma \ref{Ch} to obtain, for any $n\in \mathbb{N}_0$, that
	\begin{equation}\label{BEMLS}
		\mathbb{E}\left[ |Y(t_n) - Y_h^B(t_n)|^2 \right] \leq Che^{Ct_n}.
	\end{equation} 
	Moreover, we have $E_h(t)=t_n$, $n\in\{1,\cdots N\}$, by its definition in  \eqref{Ed}. Then, a combination of \eqref{BEMLS} and Lemmas \ref{shibian} and \ref{qw} yields
	\begin{align}\label{2}
		\mathbb{E}\left[\left|Y(E_h(t))-Y_h^B(E_h(t))\right|^2\right]&\leq Ch\mathbb{E}\left[e^{CE_h(t)}\right]\leq Ch\mathbb{E}\left[e^{CE(t)}\right]\leq Che^{Ct}\leq Che^{CT}.
	\end{align}
	Substituting \eqref{1} and \eqref{2} into \eqref{0} yields the desired result.
	This completes the proof.
\end{proof}
\subsection{Proof of Theorem \ref{PEM1}}\label{sec6.5}
For short notation, define, for any $x\in\mathbb{R}^d$, $n\in\mathbb{N}_0$, $h\in(0,1]$,
\begin{equation}\label{Phi}
	\Phi(x,t_{n},h):=\kappa(x)+h b(t_{n}, \kappa(x)) +g(t_{n}, \kappa(x))\Delta W_{n}.
\end{equation}
Then, PEM \eqref{PEM} can be rewritten as
$$
\begin{cases}
	Y_h^P(t_{n+1})=\Phi(Y_h^P(t_n),t_n,h),\\
	Y_h^P(0)=Y(0).
\end{cases}
$$
Denote by $\mathcal{P}_i(x):=x-\mathbb{E}\left[x\mid\mathcal{F}_{t_{i-1}}\right]$. The following proposition provides a convergence result for PEM \eqref{PEM}, which is  a special case of \cite[Lemma 3.5]{MR3493491}. 
\begin{proposition}\cite[Lemma 3.5]{MR3493491}\label{PEMSL}
	Suppose Assumptions \ref{A1}-\ref{A4} hold with $\eta>3$. Let $h\leq 1$. Then, we have, for every $n\in \mathbb{N}$, that
	\begin{align}\label{sum}
	&\mathbb{E}\left[\left|Y\left(t_n\right)-Y_h^P\left(t_n\right)\right|^2\right] 
	\nonumber\\&\leq \sum_{i=1}^n   \left(1+h^{-1}\right)\mathbb{E}\left[\left|\mathbb{E}\left[Y\left(t_i\right)-\Phi\left(Y\left(t_{i-1}\right), t_{i-1}, h\right) \mid \mathcal{F}_{t_{i-1}}\right]\right|^2\right]e^{Ct_n} \nonumber\\
	&\quad+\frac{\eta-1}{\eta-3} \sum_{i=1}^n\mathbb{E}\left[\left|\mathcal{P}_i\left(Y\left(t_i\right)-\Phi\left(Y\left(t_{i-1}\right), t_{i-1}, h\right)\right)\right|^2\right] e^{Ct_n}.
\end{align}
\end{proposition}
\begin{lemma}\label{PEMSLJ}
	Suppose Assumptions \ref{A1}-\ref{A4} hold with $p^\star\geq 6\gamma-4$. Then, we have, for any $i\in\mathbb{N}$, that
	\begin{equation}\label{X-P}
		\mathbb{E}\left[\left|\mathbb{E}\left[Y\left(t_i\right)-\Phi\left(Y\left(t_{i-1}\right), t_{i-1}, h\right) \mid \mathcal{F}_{t_{i-1}}\right]\right|^2\right]\leq Ch^3e^{Ct_i},
	\end{equation}
	\begin{equation}\label{ID}
		\mathbb{E}\left[\left|\mathcal{P}_i(Y\left(t_i\right)-\Phi\left(Y\left(t_{i-1}\right), t_{i-1}, h\right))\right|^2\right]\leq  Ch^2e^{Ct_i}.
	\end{equation}
	
\end{lemma}
\begin{proof}
	See Appendix \ref{app1}.
\end{proof}

\begin{proof}[Proof of Theorem \ref{PEM1}]
	Substituting \eqref{X-P} and \eqref{ID} in Lemma \ref{PEMSLJ} into \eqref{sum} and noting that $t_n=nh$ yield
	\begin{align}\label{PEMHE}
		\mathbb{E}\left[\left|Y(t_n)-Y^P_h(t_n)\right|^2\right]&\leq\sum_{i=1}^{n}\left[(1+h^{-1})Ch^3e^{Ct_i}\right]e^{Ct_n}+\sum_{i=1}^{n}Ch^2Ce^{Ct_i}e^{Ct_n}\nonumber\\
		&\leq Cht_ne^{Ct_n}\leq Che^{Ct_n},
	\end{align}
	where the last inequality holds due to $x\leq e^x$ for all $x\in\mathbb{R}$.
	By using \eqref{PEMHE} and by following the same argument as in the proof of Theorem \ref{BEMT}, we obtain, for any $t\in[0,T]$, that
	\begin{align}
		\mathbb{E}\left[\left|X(t)-Y_h^P(E_h(t))\right|^2\right]\leq Che^{CT}\nonumber.
	\end{align}
	This completes the proof.
\end{proof}
\bmhead{Acknowledgements}
The authors are grateful to Professor Ziheng Chen for his helpful suggestions.
\section*{Declarations}
\bmhead{Funding}This work was supported by the Natural Science Foundation of Changsha (No. kq2502101), the National Natural Science Foundation of China (Nos. 12471394, 12371417), and the Guangzhou-HKUST (GZ) Joint Funding Program (No. 2025A03J3322).
\bmhead{Competing interests} The authors have not disclosed any conflict of interest.
\bmhead{Authors' contributions}
\textbf{Shuai Wang}: Conceptualization, Methodology, Software, Investigation, Writing - original draft.
\textbf{Yuanling Niu}: Supervision, Writing - review $\&$ editing, Funding acquisition. 
\textbf{Ying Zhang}: Supervision, Writing - review $\&$ editing, Funding acquisition.
\begin{appendices}
\section{Proof of results in Section \ref{sec4}}\label{app2}
\begin{proof}[Proof of Proposition \ref{E1p}]
	For any $x,y\in \mathbb{R}$
	$$
	\begin{aligned}
		|b(t,x)-b(t,y)|&=|x^2-2x^5-y^2+2y^5|\\&=\left|(x-y)(x+y)-2(x-y)(x^4+x^3y+x^2y^2+xy^3+y^4)\right|\\
		&\leq|x-y|\left|\left(x+y\right)-2(x^4+x^3y+x^2y^2+xy^3+y^4)\right|\\
		&\leq|x-y|\left[\left|x+y|+2|x^4+x^3y+x^2y^2+xy^3+y^4\right|\right]\\
		&\leq 9\left(1+|x|^4+|y|^4\right)\left|x-y\right|,
	\end{aligned}
	$$
	where the last inequality holds due to Young's inequality. It is clear that  $|b(t,0)|=0\leq K$ and $|g(t,0)|=0\leq K$ for all $t\geq0$ and $K>0$. This implies  Assumption \ref{A1} holds with $\gamma=5$ and $K\geq9$.
	
	Moreover, for any $\eta>1$, applying the inequality $-2(x^3y+xy^3)\leq x^4+y^4+2x^2y^2$ yields
	$$
	\begin{aligned}
		&(x-y)\left(x^2-2x^5-y^2+2y^5\right)+\frac{\eta-1}{2}\left|x^2-y^2\right|^2\\
		&\leq|x-y|^2\left[x+y-2\left(x^4+xy^3+x^2y^2+x^3y+y^4\right)+\frac{\eta-1}{2}\left|x+y\right|^2\right]\\
		&\leq\left|x-y\right|^2\left[x+y-x^4-y^4+\left(\eta-1\right)\left|x+y\right|^2\right]\\
		&\leq \left(\frac{1}{2}+\frac{1}{2}\left(\eta-\frac{1}{2}\right)^2\right)\left|x-y\right|^2.
	\end{aligned}
	$$
	This implies that Assumption \ref{A2} holds for any $\eta>1$ with $K_1\geq\frac{1}{2}+\frac{1}{2}\left(\eta-\frac{1}{2}\right)^2$. In particular, we make the concrete choice $\eta=3$ and $K_1=13$. This ensures that for the BEM, every step size $h$ in Section \ref{li1} satisfies $h\in(0,1\wedge\frac{1}{4K_1}]$.
	
	Similarly, Assumption \ref{A3} holds for any $p^\star>1$ due to the following inequality
	$$
	\begin{aligned}
		x\left(x^2-2x^5\right)+\frac{p^\star-1}{2}x^2= x^2\left(x-2x^4+\frac{p^\star-1}{2}\right)\leq K\left(1+|x|^2\right).
	\end{aligned}
	$$
	As TCSDE \eqref{fc1} is an autonomous equation, Assumption \ref{A4} holds.
	This completes the proof.
\end{proof}
\begin{proof}[Proof of Proposition \ref{E2p}]
	For any $x,y\in \mathbb{R}^2$,
	\begin{equation}\label{E2A1}
		|b(t,x)-b(t,y)|=|f(x)-f(y)-A(x-y)|
		\leq|f(x)-f(y)|+ |A(x-y)|.
	\end{equation}
	By applying similar arguments as in the proof of Proposition \ref{E1p}, we obtain that 
	\begin{align}\label{b_1}
		|f(x)-f(y)| &= \sqrt{\sum_{i=1}^2 |(x_i-y_i)- (x_i^3-y_i^3)|^2} \nonumber\\
		&= \sqrt{\sum_{i=1}^2 |x_i-y_i|^2 \cdot |1-(x_i^2+x_i y_i+y_i^2)|^2} \nonumber\\
		&\leq \sqrt{\sum_{i=1}^2 |x_i-y_i|^2 \cdot \left(1+\frac{3}{2}(x_i^2+y_i^2)\right)^2} \nonumber\\
		&\leq \frac{3}{2}(1+|x|^2+|y|^2) \sqrt{\sum_{i=1}^2 |x_i-y_i|^2} \nonumber\\
		&= \frac{3}{2}(1+|x|^2+|y|^2)|x-y|.
	\end{align}
	Moreover, we note that the eigenvalues of $A$ are $\lambda_1=1$ and $\lambda_2=200$. Hence, we have
	\begin{equation}\label{b_2}
		|A(x-y)|\leq\|A\|\cdot|x-y|\leq200|x-y|.
	\end{equation}
	By substituting \eqref{b_1} and \eqref{b_2} into \eqref{E2A1}, we conclude that
	$$
	|b(t,x)-b(t,y)|\leq \frac{3}{2}(1+|x|^2+|y|^2)|x-y|+200|x-y|\leq202(1+|x|^2+|y|^2)|x-y|.
	$$
	This implies Assumption \ref{A1} holds with $\gamma=3$ and $K=202$.
	
	To show Assumption \ref{A2} holds, for any $x,y\in\mathbb{R}^2$, we denote by
	$$
	z=\left[
	\begin{matrix}
		z_1\\
		z_2
	\end{matrix}
	\right]:=x-y=\left[
	\begin{matrix}
		x_1-y_1\\
		x_2-y_2
	\end{matrix}
	\right].
	$$
	We note that 
	$$
	f ( x ) - f ( y ) =
	\left[\begin{matrix}
		x_1-y_1-(x_1^3-y_1^3)\\
		x_2-y_2-(x_2^3-y_2^3)
	\end{matrix}\right]
	=\left[\begin{matrix}
		z_1-z_1(x_1^2+x_1y_1+y_1^2)\\
		z_2-z_2(x_2^2+x_2y_2+y_2^2)
	\end{matrix}\right],
	$$
	and
	$$
	A(x-y)=Az=\frac{1}{2}\left[\begin{matrix}
		201z_1-199z_2\\
		-199z_1+201z_2
	\end{matrix}\right].
	$$
	Therefore, we have
	$$
	b(t,x)-b(t,y)=f(x)-f(y)-A(x-y)=\left[\begin{matrix}
		z_1-z_1s_1-\frac{201}{2}z_1+\frac{199}{2}z_2\\
		z_2-z_2s_2+\frac{199}{2}z_1-\frac{201}{2}z_2
	\end{matrix}\right],
	$$
	where $s_1=x_1^2+x_1y_1+y_1^2$ and $s_2=x_2^2+x_2y_2+y_2^2$. Then, we obtain
	$$
	\begin{aligned}
		\left<b(t,x)-b(t,y),x-y\right>
		&= \left(-s_1-\frac{199}{2}\right)z_1^2+\frac{199}{2}z_1z_2+\left(-s_2-\frac{199}{2}\right)z_2^2+\frac{199}{2}z_1z_2\\
		&= \left(-s_1-\frac{199}{2}\right)z_1^2+\left(-s_2-\frac{199}{2}\right)z_2^2+199z_1z_2\\
		&\leq\left(-s_1-\frac{199}{2}\right)z_1^2+\left(-s_2-\frac{199}{2}\right)z_2^2+\frac{199}{2}\left(z_1^2+z_2^2\right)\\
		&=-s_1z_1^2-s_2z_2^2\\
		&\leq0.
	\end{aligned}
	$$
	Moreover, we have, for any $x,y\in\mathbb{R}^2$, that
	$$
	|g(t,x)-g(t,y)|^2=\left(\frac{1}{2}\right)^2\left[\left(x_1-y_1\right)^2+\left(x_2-y_2\right)^2\right]=\frac{1}{4}\left|x-y\right|^2.
	$$
	Then, for any $\eta>1$, we can deduce that
	$$
	\left<b(t,x)-b(t,y),x-y\right>+\frac{\eta-1}{2}\left|g(t,x)-g(t,y)\right|^2\leq K_1|x-y|^2.
	$$
	This implies that Assumption \ref{A2} holds for any $\eta>3$ with $K_1\geq\frac{\eta-1}{8}$.
	
	By using similar arguments to the above, we obtain that
	\begin{align}
		\langle b(t,x), x \rangle +\frac{p^\star - 1}{2}|g(t,x)|^2&= -\frac{ 199}{2}|x|^2 - 199x_1 x_2 - \left(x_1^4 + x_2^4\right)+\frac{p^\star - 1}{8}|x|^2\nonumber\\
		&\leq\frac{p^\star - 1}{8}\left(1+|x|^2\right)\nonumber.
	\end{align}
	Hence, Assumption~\ref{A3} holds for any $p^\star>1$ with an appropriate choice of $K\geq \frac{p^\star-1}{8}$. 
	
	Assumption~\ref{A4} holds as TCSDE~\eqref{li2} has autonomous coefficients. This completes the proof.
	
\end{proof}
\section{Proof of auxiliary results in Section \ref{sec5}}\label{app1}
\begin{proof}[Proof of Lemma \ref{jhd}]
	Recall the expression of SDE \eqref{do}. For $p>2$, by applying the elementary inequality, H\"{o}lder's inequality, and \cite[Theorem 7.1]{MR2380366}, we obtain, for any $0\leq s \leq t$, that
	$$\begin{aligned} 
		&\mathbb{E}\left[|Y(t)-Y(s)|^p\right] 
		\nonumber\\&\leq   2^{p-1} \mathbb{E} \left[\left|\int_s^tb(r, Y(r)) \mathrm{~d} r\right|^p\right]
		+2^{p-1} \mathbb{E} \left[\left|\int_s^tg(r, Y(r)) \mathrm{~d} W(r)\right|^p\right]\\
		&\leq   C|t-s|^{p-1}\mathbb{E}\left[\int_s^t |b(r, Y(r))|^p \mathrm{~d} r\right]+C|t-s|^{\frac{p-2}{2}}\mathbb{E}\left[\int_s^t |g(r, Y(r))|^p \mathrm{~d} r\right].
	\end{aligned}$$
	Applying the elementary inequality, Fubini theorem, Remark~\ref{re2.1}, and Lemma~\ref{ju} yields
	$$\begin{aligned}
		\mathbb{E}\left[\left|Y(t)-Y(s)\right|^p\right] & \leq C|(t-s)|^{p-1}\mathbb{E} \left[\int_s^t(1+|Y(r)|^{\gamma})^p\mathrm{~d} r\right]\nonumber\\
		&\quad+C|t-s|^{\frac{p-2}{2}} \mathbb{E} \left[\int_s^t(1+|Y(r)|^{\frac{\gamma+1}{2}})^p\mathrm{~d} r\right]\\
		& \leq C|t-s|^{\frac{p}{2}} e^{C t}.
	\end{aligned}
	$$
	
	For $p\in$$(1, 2]$, a similar result can be obtained by using Lemma~\ref{ju} and Jensen's inequality. Combining the results in both cases yields the desired result.
\end{proof}
\begin{proof}[Proof of Proposition \ref{jfwc}]
	To simplify the notation, we define, for any $i\in\mathbb{N}$,
	$$e_i:=Y(t_i)-Y_h^B(t_i),\quad\Delta b_i:=b(t_i,Y(t_i))-b(t_i,Y_h^B(t_i)),\quad\Delta g_{i}:=g(t_{i},Y(t_{i}))-g(t_{i},Y_h^B(t_{i})) .$$
	Recall the expressions of SDE \eqref{do}, the corresponding BEM \eqref{BEM}, and $R_i$ in \eqref{R_i}, we obtain that
	$$
	e_i=e_{i-1}+h\Delta b_i+\Delta g_{i-1}\Delta W_{i-1}+R_i.
	$$
	By using the identity
	$
	|u|^2-|v|^2+|u-v|^2=2\left<u,u-v\right>
	$
	with $u=e_i,v=e_{i-1}$, and by taking expectations on both sides, we obtain
	\begin{align}\label{zong}
		&\mathbb{E}\left[|e_i|^2\right]-\mathbb{E}\left[|e_{i-1}|^2\right]+\mathbb{E}\left[|e_i-e_{i-1}|^2\right]\nonumber\\&=2\mathbb{E}\left[\left<e_i,e_i-e_{i-1}\right>\right]\nonumber\\
		&=2\mathbb{E}\left[\left<e_i,h\Delta b_i+\Delta g_{i-1}\Delta W_{i-1}+R_i\right>\right]\nonumber\\
		&=2h\mathbb{E}\left[\left<e_i,\Delta b_i\right>\right]+
		2\mathbb{E}\left[\left<e_i-e_{i-1},\Delta g_{i-1}\Delta W_{i-1}+R_i\right>\right]\nonumber\\&\quad+
		2\mathbb{E}\left[\left<e_{i-1},\Delta g_{i-1}\Delta W_{i-1}+R_i\right>\right]\nonumber\\
		&=:I_1^i+I_2^i+I_3^i.
	\end{align}
	To obtain an upper bound for $I_1^i$, by applying Assumption \ref{A2}, we have that
	\begin{equation}\label{I1}
		I_1^i\leq 2hK_1\mathbb{E}\left[|e_i|^2\right]-(\eta-1)h\mathbb{E}\left[|\Delta g_i|^2\right].
	\end{equation}
	To establish an upper bound for $I_2^i$, an application of the elementary inequality and Young's inequality yields
	\begin{align}\label{I2}
		I_2^i&\leq \mathbb{E}\left[|e_i-e_{i-1}|^2\right]+\mathbb{E}\left[|\Delta g_{i-1}\Delta W_{i-1}+R_i|^2\right]\nonumber\\
		&=\mathbb{E}\left[|e_i-e_{i-1}|^2\right]+h\mathbb{E}\left[|\Delta g_{i-1}|^2\right]+\mathbb{E}\left[|R_i|^2\right]+2\mathbb{E}\left[\left<\Delta g_{i-1}\Delta W_{i-1},R_i\right>\right]\nonumber\\
		&\leq \mathbb{E}\left[|e_i-e_{i-1}|^2\right]+h\mathbb{E}\left[|\Delta   g_{i-1}|^2\right]+\mathbb{E}\left[|R_i|^2\right]+(\eta-2)h\mathbb{E}\left[|\Delta g_{i-1}|^2\right]+\frac{1}{\eta-2}\mathbb{E}\left[|R_i|^2\right]\nonumber\\
		&=\mathbb{E}\left[|e_i-e_{i-1}|^2\right]+(\eta-1)h\mathbb{E}\left[|\Delta g_{i-1}|^2\right]+\frac{\eta-1}{\eta-2}\mathbb{E}\left[|R_i|^2\right].
	\end{align}
	To upper bound $I_3^i$, we first note that
	$$
	\begin{aligned}
		\mathbb{E}\left[\left<e_{i-1},\Delta g_{i-1}\Delta W_{i-1}\right>\right]
		&=\mathbb{E}\left[\mathbb{E}\left[\left<e_{i-1},\Delta g_{i-1}\Delta W_{i-1}\right>\right]|\mathcal{F}_{t_{i-1}}\right]\nonumber\\&=\mathbb{E}\left[\left<e_{i-1},\mathbb{E}\left[\Delta g_{i-1}\Delta W_{i-1}|\mathcal{F}_{t_{i-1}}\right]\right>\right]\nonumber\\&=0.
	\end{aligned}
	$$
	Then, by using the properties of conditional expectation and the elementary inequality, we obtain that
	\begin{align}\label{I3}
		I_3^i&=2\mathbb{E}\left[\left<e_{i-1},R_i\right>\right]\nonumber\\
		&=2\mathbb{E}\left[\mathbb{E}\left[\left<e_{i-1},R_i\right>|\mathcal{F}_{t_{i-1}}\right]\right]\nonumber\\
		&=2\mathbb{E}\left[\left<e_{i-1},\mathbb{E}[R_i|\mathcal{F}_{t_{i-1}}]\right>\right]\nonumber\\
		&\leq h\mathbb{E}\left[|e_{i-1}|^2\right]+h^{-1}\mathbb{E}\left[|\mathbb{E}[R_i|\mathcal{F}_{t_{i-1}}]|^2\right].
	\end{align}
	Substituting \eqref{I1}-\eqref{I3} into \eqref{zong} yields
	$$
	\begin{aligned}
		\mathbb{E}\left[|e_i|^2\right]-\mathbb{E}\left[|e_{i-1}|^2\right]&\leq 2hK_1\mathbb{E}\left[|e_i|^2\right]+(\eta-1)h\left(\mathbb{E}\left[|\Delta g_{i-1}|^2\right]-\mathbb{E}\left[|\Delta g_i|^2\right]\right)\\
		&\quad+h\mathbb{E}\left[|e_{i-1}|^2\right]+h^{-1}\mathbb{E}\left[|\mathbb{E}[R_i|\mathcal{F}_{t_{i-1}}]|^2\right]+\frac{\eta-1}{\eta-2}\mathbb{E}\left[|R_i|^2\right].
	\end{aligned}
	$$
	By noticing $\mathbb{E}[|e_0|^2]=0$ and $\mathbb{E}[|\Delta g_0|^2]=0$, we obtain
	\begin{align}
		\mathbb{E}\left[|e_n|^2\right]&=\sum_{i=1}^{n}\left(\mathbb{E}\left[|e_i|^2\right]-\mathbb{E}\left[|e_{i-1}|^2\right]\right)\nonumber\\
		&\leq 2hK_1\sum_{i=1}^{n}\mathbb{E}\left[|e_i|^2\right]+(\eta-1)h\left(\mathbb{E}\left[|\Delta g_0|^2\right]-\mathbb{E}\left[|\Delta g_n|^2\right]\right)\nonumber\\
		&\quad+h\sum_{i=1}^{n}\mathbb{E}\left[|e_{i-1}|^2\right]+h^{-1}\sum_{i=1}^{n}\mathbb{E}\left[|\mathbb{E}\left[R_i|\mathcal{F}_{t_{i-1}}\right]|^2\right]+\frac{\eta-1}{\eta-2}\sum_{i=1}^{n}\mathbb{E}\left[|R_i|^2\right]\nonumber\\
		&\leq 2hK_1 \mathbb{E}\left[|e_n|^2\right]+ 2hK_1\sum_{i=1}^{n-1}\mathbb{E}[|e_i|^2]+h\sum_{i=1}^{n}\mathbb{E}[|e_{i-1}|^2]\nonumber\\
		&\quad+h^{-1}\sum_{i=1}^{n}\mathbb{E}\left[|\mathbb{E}\left[R_i|\mathcal{F}_{t_{i-1}}\right]|^2\right]+\frac{\eta-1}{\eta-2}\sum_{i=1}^{n}\mathbb{E}\left[|R_i|^2\right].\nonumber
	\end{align}
Since $h\in(0,1\wedge\frac{1}{4K_1}]$, we obtain $2hK_1\leq1/2$ and consequently $1-2hK_1\geq1/2$. By rearranging terms to isolate $\mathbb{E}\left[|e_n|^2\right]$, we obtain that
	$$
	\begin{aligned}
		\mathbb{E}\left[|e_n|^2\right]&\leq \frac{1+2K_1}{1-2hK_1}h\sum_{i=1}^{n-1}\mathbb{E}\left[|e_{i}|^2\right] +\frac{\eta-1}{(\eta-2)(1-2hK_1)}\sum_{i=1}^{n}\mathbb{E}\left[|R_i|^2\right]\\&\quad+ h^{-1}\frac{1}{1-2hK_1}\sum_{i=1}^{n}\mathbb{E}\left[|\mathbb{E}\left[R_i|\mathcal{F}_{t_{i-1}}\right]|^2\right]\\
		&\leq Ch\sum_{i=1}^{n-1}\mathbb{E}\left[|e_{i}|^2\right] +C\sum_{i=1}^{n}\mathbb{E}\left[|R_i|^2\right]+Ch^{-1}\sum_{i=1}^{n}\mathbb{E}\left[|\mathbb{E}\left[R_i|\mathcal{F}_{t_{i-1}}\right]|^2\right]\\
		&=:Ch\sum_{i=1}^{n-1}\mathbb{E}\left[|e_{i}|^2\right]+b_n,
	\end{aligned}
	$$
	where $b_n:=C\left(\sum_{i=1}^{n}\mathbb{E}\left[|R_i|^2\right]
	+h^{-1}\sum_{i=1}^{n}\mathbb{E}\left[|\mathbb{E}\left[R_i|\mathcal{F}_{t_{i-1}}\right]|^2\right]\right)$. By Gronwall's inequality \cite[Lemma]{MR878027}, we conclude that
	$$
	\begin{aligned}
		\mathbb{E}\left[|e_n|^2\right]&\leq b_ne^{\sum_{i=0}^{n-1}Ch}\leq b_ne^{Ct_n}.
	\end{aligned}
	$$
	This completes the proof.
\end{proof}
\begin{proof}[Proof of Lemma \ref{Ch}]
	Recall the definition of $R_i$ given in \eqref{R_i}.  By using the elementary inequality, H\"{o}lder's inequality, and It\^{o}'s isometry, we obtain, for any $i \leq n$, that
\begin{small}
	\begin{align}\label{ER}
	&\mathbb{E}\left[|R_i|^2\right]\nonumber\\
	&= \mathbb{E}\left[
	\left|
	\int_{t_{i-1}}^{t_i} b(r,Y(r)) - b(t_i,Y(t_i)) \, \mathrm{d}r
	\right.\right.\left.\left. + \int_{t_{i-1}}^{t_i} g(r,Y(r)) - g(t_{i-1},Y(t_{i-1})) \, \mathrm{d}W(r)
	\right|^2
	\right]\nonumber\\
	&\leq 2\mathbb{E}\left[\left|\int_{t_{i-1}}^{t_{i}}b(r,Y(r))-b(t_i,Y(t_i))\mathrm{~d} r\right|^2\right]\nonumber\\&\quad+2\mathbb{E}\left[\left|\int_{t_{i-1}}^{t_{i}}g(r,Y(r))-g(t_{i-1},Y(t_{i-1}))\mathrm{~d} W(r)\right|^2\right]\nonumber\\
	&\leq 2\mathbb{E}\left[h\int_{t_{i-1}}^{t_{i}}|b(r,Y(r))-b(t_i,Y(t_i))|^2\mathrm{~d} r\right]\nonumber\\&\quad+2\mathbb{E}\left[\int_{t_{i-1}}^{t_{i}}|g(r,Y(r))-g(t_{i-1},Y(t_{i-1}))|^2\mathrm{~d} r\right] \nonumber\\
	&=:M_1^i+M_2^i.		
\end{align}
\end{small}
	To establish an upper bound for $M_1^i$, we apply the elementary inequality, Assumptions~\ref{A1} and \ref{A4}, Fubini's theorem, and H\"{o}lder's inequality to obtain that
	\begin{align}\label{M10}
		M_1^i &\leq 4\mathbb{E}\left[h\int_{t_{i-1}}^{t_{i}}\left|b(r,Y(r))-b(t_i,Y(r))\right|^2+|b(t_i,Y(r))-b(t_i,Y(t_i))|^2\mathrm{~d} r\right] \nonumber \\
		&\leq 4K^2h\mathbb{E}\left[\int_{t_{i-1}}^{t_{i}}\left(1+|Y(r)|^\gamma\right)^2|t_i-r|\mathrm{d}r\right]  \nonumber\\&\quad+4K^2h\mathbb{E}\left[\int_{t_{i-1}}^{t_{i}}\left(1+|Y(r)|^{\gamma-1}+|Y(t_i)|^{\gamma-1}\right)^2|Y(r)-Y(t_i)|^2\mathrm{~d} r\right] \nonumber \\
		&\leq 8K^2h\int_{t_{i-1}}^{t_{i}}\left(1+\mathbb{E}\left[|Y(t)|^{2\gamma}\right]\right)|t_i-r|\mathrm{~d}r \nonumber\\ &\quad+4hK_1^2\int_{t_{i-1}}^{t_{i}}\left(\mathbb{E}\left[(1+|Y(r)|^{\gamma-1}+|Y(t_i)|^{\gamma-1})^\frac{4\gamma-2}{\gamma-1}\right]\right)^{\frac{\gamma-1}{2\gamma-1}} \nonumber\\&\qquad\times\left(\mathbb{E}\left[|Y(r)-|Y(t_i)|^{\frac{4\gamma-2}{\gamma}}\right]\right)^{\frac{\gamma}{2\gamma-1}}\mathrm{~d} r\nonumber\\
		&\leq Ch\int_{t_{i-1}}^{t_{i}}\left(1+\sup_{0 \leq t \leq t_i}\mathbb{E}\left[|Y(t)|^{2\gamma}\right]\right)|t_i-r|\mathrm{~d} r \nonumber \\
		&\quad +Ch\int_{t_{i-1}}^{t_{i}}\left(1+2\sup_{0 \leq t \leq t_i}\mathbb{E}\left[|Y(t)|^{4\gamma-2}\right]\right)^{\frac{\gamma-1}{2\gamma-1}}\left(\mathbb{E}\left[|Y(r)-Y(t_i)|^{\frac{4\gamma-2}{\gamma}}\right]\right)^{\frac{\gamma}{2\gamma-1}}\mathrm{~d} r. 
	\end{align}
	Since $p^\star \geq 4\gamma - 2$ and $\gamma > 1$, we have $p^\star>2\gamma$ and $1<\frac{4\gamma-2}{\gamma} <\frac{p^\star}{\gamma}$ . Thus, by applying Lemmas~\ref{ju} and \ref{jhd}, we obtain that
	\begin{align}\label{m1}
		M_1^i &\leq Ch \int_{t_{i-1}}^{t_{i}}\left(1+Ce^{Ct_i}\right)|t_i-r|\mathrm{~d} r+Ch\int_{t_{i-1}}^{t_{i}}\left(1+Ce^{Ct_i}\right)^{\frac{\gamma-1}{2\gamma-1}}Ce^{Cr}|r-t_i|\mathrm{~d} r\nonumber\\
		&\leq Ch\left(1+Ce^{Ct_i}\right)\int_{t_{i-1}}^{t_{i}}|t_i-r|\mathrm{~d}r\nonumber\\
		&\leq Ch^3\left(1+Ce^{Ct_i}\right).
	\end{align}
	By using the same arguments as in \eqref{M10} and \eqref{m1}, we obtain that
	\begin{equation}\label{m2}
		M_2^i\leq Ch^2\left(1+Ce^{Ct_i}\right).
	\end{equation}
	Substituting \eqref{m1} and \eqref{m2} into \eqref{ER} yields
	\begin{equation}\label{Ri}
		\mathbb{E}\left[|R_i|^2\right]\leq Ch^2\left(1+Ce^{Ct_i}\right).
	\end{equation}
	Moreover, to obtain an upper bound for $\mathbb{E}[|\mathbb{E}[R_i|\mathcal{F}_{t_{i-1}}]|^2]$, we note that
	\begin{small}
		$$
		\begin{aligned}
			&\mathbb{E}\left[R_i|\mathcal{F}_{t_{i-1}}\right]\nonumber\\
			&=\mathbb{E}\left[\int_{t_{i-1}}^{t_{i}}b(r,Y(r))-b(t_i,Y(t_i))\mathrm{~d} r+\int_{t_{i-1}}^{t_{i}}g(r,Y(r))-g(t_i,Y(t_i))\mathrm{~d} W(r)\mid\mathcal{F}_{t_{i-1}}\right]\\
			&=\mathbb{E}\left[\int_{t_{i-1}}^{t_{i}}b(r,Y(r))-b(t_i,Y(t_i))\mathrm{~d} r\mid\mathcal{F}_{t_{i-1}}\right].
		\end{aligned}   
		$$
	\end{small}
	Thus, by applying Jensen's inequality, Assumptions \ref{A1}, \ref{A4}, Lemmas~\ref{ju}, and \ref{jhd}, we obtain that
	\begin{align}\label{Erf}
		\mathbb{E}\left[\left|\mathbb{E}\left[R_i|\mathcal{F}_{t_{i-1}}\right]\right|^2\right]&=\mathbb{E}\left[\left|\mathbb{E}\left[\int_{t_{i-1}}^{t_{i}}b(r,Y(r))-b(t_i,Y(t_i))\mathrm{~d} r\mid\mathcal{F}_{t_{i-1}}\right]\right|^2\right]\nonumber\\
		&\leq \mathbb{E}\left[\mathbb{E}\left[\left|\int_{t_{i-1}}^{t_{i}}b(r,Y(r))-b(t_i,Y(t_i))\mathrm{~d} r\right|^2\mid\mathcal{F}_{t_{i-1}}\right]\right]\nonumber\\
		&=\mathbb{E}\left[\left|\int_{t_{i-1}}^{t_{i}}b(r,Y(r))-b(t_i,Y(t_i))\mathrm{~d} r\right|^2\right]\nonumber\\
		&\leq Ch^3\left(1+Ce^{Ct_i}\right).
	\end{align}
	By using \eqref{Ri} and \eqref{Erf}, we obtain that
	$$
	\begin{aligned}
		&\sum_{i=1}^{n}\mathbb{E}\left[R_i\right]^2+h^{-1}\sum_{i=1}^{n}\mathbb{E}\left[\left|\mathbb{E}\left[R_i\mid\mathcal{F}_{t_{i-1}}\right]\right|^2\right]
		\\&\leq\sum_{i=1}^{n} Ch^2\left(1+Ce^{Ct_i}\right)+h^{-1}\sum_{i=1}^{n}Ch^3\left(1+Ce^{Ct_i}\right)\\
		&\leq Cht_n\left(1+Ce^{Ct_n}\right)\\
		&\leq Che^{Ct_n},
	\end{aligned}
	$$
	where the second inequality holds due to $t_i\leq t_n$ for all $i\leq n$, and $x\leq e^x, x\in\mathbb{R}$. This completes the proof.
\end{proof}
The following lemma provides an upper bound for the projection error.
\begin{lemma}\label{x0}
	We have, for any $x\in\mathbb{R}^d$ and $m\geq0$, that
	$$
	|x-\kappa(x)|\leq2h^m|x|^{1+2m(\gamma-1)}.
	$$
\end{lemma}
\begin{proof}
	Recall the definition of $\kappa$ in \eqref{xq} with $h\in(0,1]$ and $\alpha=\frac{1}{2(\gamma-1)}$. If $|x|\leq h^{-\alpha}$, then $\kappa(x)=x$. This immediately proves the lemma.
	If $|x|> h^{-\alpha}$, then $|\kappa(x)|\leq h^{-\alpha}\leq |x|$. Thus, we obtain, for any $m\geq0$, that
	$$
	|x-\kappa(x)|\leq \mathbf{1}_{\{|x|> h^{-\alpha}\}}2\left|x\right|\leq2h^m\left|x\right|^{1+\frac{m}{\alpha}}=2h^m\left|x\right|^{1+2m(\gamma-1)}.
	$$
	Combining the two cases completes the proof.
\end{proof}
\begin{proof}[Proof of Lemma \ref{PEMSLJ}]
	Recalling the definition of $\Phi$ in \eqref{Phi}, it follows, for any $i\in\mathbb{N},$ that
	\begin{small}
		$$
		\begin{aligned}
			&\mathbb{E}\left[Y\left(t_i\right)-\Phi\left(Y\left(t_{i-1}\right), t_{i-1}, h\right) \mid \mathcal{F}_{t_{i-1}}\right]\nonumber\\
			&=\mathbb{E}\left[Y(t_i)-\kappa\left(Y(t_{i-1})\right)-hb\left(t_{i-1},\kappa(Y(t_{i-1}))\right)
			-\Delta W_{i-1} g(t_{i-1},\kappa(Y(t_{i-1})))\mid \mathcal{F}_{t_{i-1}}\right]\\
			&=\mathbb{E}\left[\int_{t_{i-1}}^{t_i} b(r,Y(r))-b(t_{i-1},\kappa(Y(t_{i-1})))\mathrm{~d} r+Y(t_{i-1})-\kappa(Y(t_{i-1}))\mid \mathcal{F}_{t_{i-1}}\right].
		\end{aligned}
		$$
	\end{small}
	By Jensen's inequality and the elementary inequality, we have
	\begin{small}
		\begin{align}\label{JJ}
			&\mathbb{E}\left[\left|\mathbb{E}\left[Y\left(t_i\right)-\Phi\left(Y\left(t_{i-1}\right), t_{i-1}, h\right) \mid \mathcal{F}_{t_{i-1}}\right]\right|^2\right]\nonumber\\
			&\leq\mathbb{E}\left[\left|\int_{t_{i-1}}^{t_i} b(r,Y(r))-b(t_{i-1},\kappa(Y(t_{i-1})))\mathrm{~d} r+Y(t_{i-1})-\kappa(Y(t_{i-1}))\right|^2\right]\nonumber\\
			&\leq2\mathbb{E}\left[\left|\int_{t_{i-1}}^{t_i} b(r,Y(r))-b(t_{i-1},\kappa(Y(t_{i-1})))\mathrm{~d} r\right|^2\right]+2\mathbb{E}\left[\left|Y(t_{i-1})-\kappa(Y(t_{i-1}))\right|^2\right]\nonumber\\
			&=:J^i_1+J^i_2.
		\end{align}
	\end{small}
	To upper bound $J^i_1$, the application of H\"{o}lder's inequality yields
	\begin{align}\label{J1}
		J^i_1
		\leq	& Ch\mathbb{E}\left[\int_{t_{i-1}}^{t_i} |b(r,Y(r))-b(t_{i-1},\kappa(Y(t_{i-1})))|^2\mathrm{~d} r\right]\nonumber\\
		\leq& Ch\mathbb{E}\left[\int_{t_{i-1}}^{t_i} |b(r,Y(r))-b(t_{i-1},Y(t_{i-1}))|^2\mathrm{~d} r\right]\nonumber\\\quad&+Ch^2\mathbb{E}\left[\left|b(t_{i-1},Y(t_{i-1}))-b(t_{i-1},\kappa(Y(t_{i-1})))\right|^2\right]\nonumber\\
		=	&:J^i_{11}+J^i_{12}.
	\end{align}
	We note that an upper bound for $J^i_{11}$ in \eqref{J1} can be obtained using the same argument as in \eqref{M10} and \eqref{m1}, hence, we obtain
	\begin{equation}\label{J11}
		J^i_{11} \leq C h^3e^{C t_i}.
	\end{equation}
	To estimate $J^i_{12}$ in \eqref{J1}, applying Assumption \ref{A1} and H\"{o}lder's inequality yields
	$$
	\begin{aligned}
		J^i_{12}&\leq Ch^2\mathbb{E}\left[\left(1+|Y(t_{i-1})|^{\gamma -1}+|\kappa(Y(t_{i-1})\right)|^{\gamma-1})^2\left|Y(t_{i-1})-\kappa(Y(t_{i-1}))\right|^2\right]\\
		&\leq Ch^2\left(\mathbb{E}\left[\left(1+|Y(t_{i-1})|^{\gamma-1}+|\kappa(Y(t_{i-1})\right)|^{\gamma-1})^\frac{4\gamma-2}{\gamma-1}\right]\right)^{\frac{\gamma-1}{2\gamma-1}}\nonumber\\&\quad\times\left(\mathbb{E}\left[\left|Y(t_{i-1})-\kappa(Y(t_{i-1}))\right|^{\frac{4\gamma-2}{\gamma}}\right]\right)^{\frac{\gamma}{2\gamma-1}}.
	\end{aligned}
	$$
	By using Lemmas \ref{ju} and \ref{x0} with $m=1/2$ and $x=Y(t_{i-1})$, and the fact that $|\kappa(x)|\leq |x|,\text{ for all } x\in \mathbb{R}^d$, we obtain that
	\begin{align}\label{J12}
		J^i_{12}&\leq Ch^3\left(1+2\mathbb{E}\left[|Y(t_{i-1})|^{4\gamma-2}\right]\right)^{\frac{\gamma-1}{2\gamma-1}}\left(\mathbb{E}\left[|Y(t_{i-1})|^{4\gamma-2}\right]\right)^{\frac{\gamma}{2\gamma-1}}\nonumber\\
		&\leq Ch^3\left(1+Ce^{Ct_{i-1}}\right)Ce^{Ct_{i-1}}\nonumber\\
		&\leq Ch^3e^{Ct_{i}}.
	\end{align}
	Substituting \eqref{J11} and \eqref{J12} into \eqref{J1} yields
	\begin{equation}\label{J1z}
		J^i_1\leq Ch^3e^{Ct_{i}}.
	\end{equation}
	To upper bound $J_2^i$, we apply Lemmas \ref{ju} and \ref{x0} with $m=3/2$ and $x=Y(t_{i-1})$ to obtain
	\begin{align}\label{J2}
		J^i_2=2\mathbb{E}\left[|Y(t_{i-1})-\kappa(Y(t_{i-1}))|^2\right]
		\leq Ch^3\mathbb{E}\left[|Y(t_{i-1})|^{6\gamma-4}\right]
		\leq Ch^3e^{Ct_i}.
	\end{align}
	Substituting \eqref{J1z} and \eqref{J2} into \eqref{JJ} yields the desired result in \eqref{X-P}.	
	
	Next, we prove the inequality \eqref{ID}.  Recall $\mathcal{P}_i(x):=x-\mathbb{E}\left[x\mid\mathcal{F}_{t_{i-1}}\right]$. By using the inequality
	$$
	\mathbb{E}\left[\left|\mathcal{P}_i(Y)\right|^2\right]\leq \mathbb{E}\left[\left|Y\right|^2\right],
	$$
	we obtain, for all $i\in\mathbb{N}$, that
	$$
	\begin{aligned}
		&\mathbb{E}\left[\left|\mathcal{P}_i\left(\left(Y\left(t_i\right)-\Phi\left(Y\left(t_{i-1}\right), t_{i-1}, h\right)\right)\right)\right|^2\right]\nonumber\\
		&\leq \mathbb{E}\left[\left|Y\left(t_i\right)-\Phi\left(Y\left(t_{i-1}\right), t_{i-1}, h\right)\right|^2\right]\\
		&=\mathbb{E}\left[\left| Y(t_{i-1})-\kappa(Y(t_{i-1}))+\int_{t_{i-1}}^{t_i}b(r,Y(r))-b\left(t_{i-1},\kappa(Y(t_{i-1}))\right)\mathrm{~d} r \right.\right.\\
		&\quad+\left. \left.\int_{t_{i-1}}^{t_i}g(r,Y(r))-g(t_{i-1},\kappa(Y(t_{i-1})))\mathrm{~d} W(r)\right|^2\right].
	\end{aligned}
	$$
	By using the elementary inequality and It\^{o}'s isometry, we have that
	\begin{align}\label{Q}
		&\mathbb{E}\left[\left|\mathcal{P}_i(\left(Y\left(t_i\right)-\Phi\left(Y\left(t_{i-1}\right), t_{i-1}, h\right)\right))\right|^2\right]\nonumber\\
		&\leq C\mathbb{E}\left[|Y(t_{i-1})-\kappa(Y(t_{i-1}))|^2\right]\nonumber\\
		&\quad +C\mathbb{E}\left[\left|\int_{t_{i-1}}^{t_i}b(r,Y(r))-b(t_{i-1},\kappa(Y(t_{i-1})))\mathrm{~d} r\right|^2\right]\nonumber\\
		&\quad+C\mathbb{E}\left[\left|\int_{t_{i-1}}^{t_i}g(r,Y(r))-g(t_{i-1},\kappa(Y(t_{i-1})))\mathrm{~d} W(r)\right|^2\right]\nonumber\\
		&=C\mathbb{E}\left[|Y(t_{i-1})-\kappa(Y(t_{i-1}))|^2\right]\nonumber\\&\quad+C\mathbb{E}\left[\left|\int_{t_{i-1}}^{t_i}b(r,Y(r))-b(t_{i-1},\kappa(Y(t_{i-1})))\mathrm{~d} r\right|^2\right].
	\end{align}
	By applying the same arguments as those in \eqref{JJ}-\eqref{J12}, we obtain, for any $i\in\mathbb{N}$, that
	$$
	\mathbb{E}\left[\left|\mathcal{P}_i\left(\left(Y\left(t_i\right)-\Phi\left(Y\left(t_{i-1}\right), t_{i-1}, h\right)\right)\right)\right|^2\right]\leq Ch^2e^{Ct_i}.
	$$
	This completes the proof.
\end{proof}
\end{appendices}
\bibliography{sn-bibliography}

\end{document}